\documentclass{article}

\usepackage{amsmath}
\usepackage{amsthm}
\usepackage{stmaryrd}

\usepackage[comma,numbers,square]{natbib}

\usepackage[labelfont = bf]{caption}

\usepackage{comment}

\usepackage{tikz}
\usepackage[all]{xy}

\usepackage{amsfonts}
\usepackage{amssymb}

\usepackage[colorlinks]{hyperref}
\usepackage[capitalize]{cleveref}
\crefname{section}{\textsection}{\textsection}
\Crefname{section}{\textsection}{\textsection}
\crefname{figure}{Fig.}{Figs.}
\Crefname{figure}{Fig.}{Figs.}
\usepackage{url}
\newcommand{\email}[1]{\href{mailto:#1}{\nolinkurl{#1}}}
\newcommand{\OEIS}[1]{\href{http://oeis.org/#1}{\nolinkurl{#1}}}
\newcommand{\doi}[1]{\textsc{doi}: \href{http://dx.doi.org/#1}{\nolinkurl{#1}}}

\usepackage[margin = 1.2in]{geometry}

\usepackage{enumitem}
\setlist[enumerate, 1]{label = (\roman*), font = \upshape}
\setlist[itemize, 2]{label = {$\circ$}}

\theoremstyle{plain}
\newtheorem{theorem}{Theorem}
\newtheorem{corollary}[theorem]{Corollary}
\newtheorem{lemma}[theorem]{Lemma}

\theoremstyle{definition}
\newtheorem{conjecture}[theorem]{Conjecture}

\newtheorem{definition}[theorem]{Definition}
\newtheorem{example}[theorem]{Example}

\theoremstyle{remark}

\DeclareMathOperator{\F}{\mathbb{F}}                            
\DeclareMathOperator{\Z}{\mathbb{Z}}                            
\DeclareMathOperator{\Q}{\mathbb{Q}}                            
\DeclareMathOperator{\R}{\mathbb{R}}                            

\begin{document}
\title{Classifying Subatomic Domains}
\author{
    Noah Lebowitz-Lockard \\
    Department of Mathematics\\
    University of Georgia \\
    \texttt{\email{noah.lebowitzl25@uga.edu}}
    }
\maketitle

\begin{abstract}
In general, ring theory is focused on atomic rings, i.e. rings in which every element has some factorization into irreducible elements. In a recent paper of Boynton and Coykendall \cite{BC}, the two authors introduce two properties that are slightly weaker than atomicity, which they call ``almost atomicity" and ``quasiatomicity". In this paper, we classify various subatomic properties and show that they are all distinct.
\end{abstract}

\section{Introduction}

Much of ring theory is either focused on UFDs or arbitrary commutative rings. This dichotomy leaves out many properties that are weaker than unique factorization, but are still important in their own right. One of the first attempts to rectify this problem was a paper of D. D. Anderson, D. F. Anderson, and M. Zafrullah \cite{AAZ}, in which the authors classify rings in which every element can be factored, but not necessarily uniquely. In this paper, we classify various types of subatomic domains, i.e. those which do not possess factorization, but still satisfy similar, weaker conditions.

Figure $1$ shows how all of the properties in this paper relate to each other. The central goal of this paper is to establish that Figure $1$ is ``complete" in the sense that there are no additional implications that we did not include. This paper mostly completes that goal with a few exceptions, discussed at the end of this introduction.

\begin{figure}[h]
\centering
\begin{tikzpicture}
\node (UFD) at (-1, 1) {Unique factorization domain};
\node (A) at (0, 0) {Atomic};
\node (SA) at (4, 0) {Semi-atomic};
\node (AA) at (-1.8, -2) {Almost atomic};
\node (QA) at (2.2, -2) {Quasi-atomic};
\node (F) at (0, -4) {Furstenberg};
\node (SF) at (4, -4) {Semi-Furstenberg};
\node (AF) at (-1.8, -6) {Almost Furstenberg};
\node (QF) at (2.2, -6) {Quasi-Furstenberg};
\node (NA) at (3.2, -7) {Not antimatter};
\node (ID) at (4.2, -8) {Integral domain};

\draw[thick, ->] (UFD) -- (A);

\draw[thick, ->] (A) -- (SA);
\draw[thick, ->] (SA) -- (AA);
\draw[thick, ->] (AA) -- (QA);

\draw[thick, ->] (A) -- (F);
\draw[thick, ->] (SA) -- (SF);
\draw[thick, ->] (AA) -- (AF);
\draw[thick, ->] (QA) -- (QF);

\draw[thick, ->] (F) -- (SF);
\draw[thick, ->] (F) -- (AF);
\draw[thick, ->] (SF) -- (QF);
\draw[thick, ->] (AF) -- (QF);

\draw[thick, ->] (QF) -- (NA);
\draw[thick, ->] (NA) -- (ID);
\end{tikzpicture}
\caption{Various classes of rings.}
\label{fig:Bill}
\end{figure}
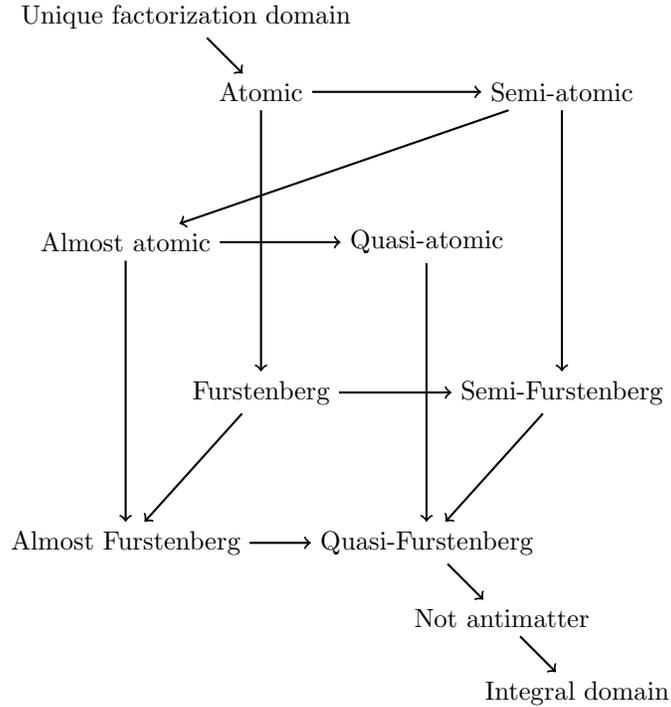

\begin{conjecture} None of the arrows in Figure $1$ are bidirectional.
\end{conjecture}

\begin{conjecture} Figure $1$ requires no additional arrows.
\end{conjecture}

The properties we investigate are of two types. ``Subatomic" domains possess the property that for $\alpha$ in the ring, there exists some $\beta$ satisfying a certain set of conditions such that $\alpha\beta$ can be expressed as the product of irreducible elements. In a ``sub-Furstenberg" domain, for every element $\alpha$, there exists some $\beta$ satisfying a certain set of conditions such that $\alpha\beta$ has an irreducible divisor that does not divide $\beta$.

Given that Figure $1$ contains so many properties and arrows, it would seem as though we have to prove a lot of results. In fact, we can reduce the number of results we have to prove by establishing domains that have very specific sets of properties. For example, a semi-atomic domain that is not Furstenberg is also a semi-atomic domain that is not atomic and a semi-Furstenberg domain that is not Furstenberg. With this in mind, we reduce the proofs of Conjectures $1$ and $2$ to six constructions.

\begin{enumerate}
\item{A semi-atomic domain that is not Furstenberg.}
\item{An almost atomic domain that is not semi-Furstenberg.}
\item{A quasi-atomic domain that is not almost Furstenberg.}
\item{A semi-Furstenberg domain that is not almost Furstenberg.}
\item{A Furstenberg domain that is not quasi-atomic.}
\item{A domain that is neither quasi-Furstenberg nor antimatter.}
\end{enumerate}

Of the six statements listed above, we show all of them with the exceptions of (iii) and (iv).

\section{Generalizations of Atomicity}

For a ring $D$, $D^*$ is the ring of units. Of the definitions in this section, ``semi-atomic'' is the only one that is original to this article.

\begin{definition}
An \emph{atomic domain}, or ``factorization domain'' is an integral domain in which every non-unit element is a product of irreducible elements \cite{Co}.
\end{definition}

\begin{definition}
An integral domain $D$ is \emph{semi-atomic} if there exists some $\beta \in D$ such that for all $\alpha \in D\backslash D^*$, $\alpha\beta$ is a product of irreducible elements.
\end{definition}

\begin{definition}
An integral domain $D$ is \emph{almost atomic} if for all $\alpha \in D\backslash D^*$, there exist irreducible elements $\pi_1, \pi_2, \ldots, \pi_n \in D$ such that $\alpha\pi_1 \pi_2 \ldots \pi_n$ is a product of irreducible elements \cite{BC}.
\end{definition}

\begin{definition}
An integral domain $D$ is \emph{quasi-atomic} if for all $\alpha \in D\backslash D^*$, there exists some $\beta \in D$ such that $\alpha\beta$ is a product of irreducible elements \cite{BC}.
\end{definition}

We want to show that Figure $1$ shows us the ``subatomic" properties in the correct order. The only nontrivial part is the following result.

\begin{lemma} All semi-atomic domains are almost atomic.
\end{lemma}

\begin{proof} Let $R$ be a semi-atomic domain. By definition, there exists some element $\beta$ such that $\alpha\beta$ is a product of irreducible elements for all non-unit $\alpha \in R$. If $\beta$ were a unit, then $\alpha\beta$ would only be a product of irreducible elements whenever $\alpha$ is, implying that $R$ is atomic. Suppose $\beta$ is not a unit. Then, $\beta^2$ is a product of irreducible elements because it is the product of $\beta$ and a non-unit element of $R$. If $\alpha$ is not a unit, then neither is $\alpha\beta$. So, $\alpha\beta^2 = (\alpha\beta)\beta$. Multiplying any non-unit element by $\beta^2$ produces a product of irreducible elements. Because $\beta^2$ is the product of irreducible elements, $R$ is almost atomic.
\end{proof}

At this point, we have not shown that the four properties in this section are distinct. We close this section with two simple examples showing that atomicity is stronger than almost atomicity, which is in turn stronger than quasi-atomicity.

\begin{example} We construct an almost atomic domain that is not atomic. Let $D = \Z[x] + x^2 \Q[x]$. (Another representation is $\Z + \Z x + x^2 \Q[x]$.) Let $f(x)$ be an element of $D$ in which the smallest exponent is $qx^k$ with $q \in \Q \backslash \Z$. By definition, $k > 1$. So, $f(x)$ is a multiple of every rational prime. But, the product of every rational prime is not an element of $D$. Therefore, $f(x)$ is not atomic.

Let $q = m/n$, where $m$ and $n$ are relatively prime integers. We see that $nf(x) = x^{k - 1} g(x)$, where $g(x)$ is a polynomial in which the term with the smallest exponent is $mx$. $x^{k - 1}$ is atomic because $x$ is irreducible. Let $g(x) = g_1 (x) \ldots g_k (x)$, where each $g_i (x)$ is an element of $D$. We show that $k$ cannot be arbitrarily large. If the product of non-units is $mx$, then there are at most $\omega(m) + 1$ of those non-units, where $\omega(m)$ is the number of (not necessarily distinct) prime factors of $n$. $1$ is a term of all but at most $\omega(m) + 1$ of the $g_i$'s. The rest of them have positive degree. There are at most $\textrm{deg} (g)$ of these terms. $k \leq \omega(m) + \textrm{deg} (g) + 1$. $nf(x)$ is atomic and $D$ is almost atomic.
\end{example}

\begin{example} We construct a quasi-atomic domain that is not almost atomic. Let $D = \Z[x] + x^2 \R[x]$. Let $f(x)$ be an element of $D$ in which the smallest coefficient is $rx^k$ with $r \in \R \backslash \Q$. Once again, $k > 1$ and $f(x)$ is a multiple of every rational prime. Because the rational primes are themselves irreducible in $D$, $f(x)$ is not. Hence, any element $f \in D$ in which the term with the smallest coefficient is $rx^k$ with $r \in \R \backslash \Q$ is not irreducible. Because $r$ is not a product of integers, $f(x)$ is not atomic.

We can also show that we cannot multiply $f(x)$ by any atomic element and obtain another atomic element. Let $g_1 (x), \ldots, g_n (x) \in D$ be atomic. Let the first terms of these polynomials be $a_1 x^{k_1}, \ldots, a_n x^{k_n}$. Each $a_i$ is an integer. The first term of $f(x) g_1 (x) \ldots g_n (x)$ is $ra_1 \ldots a_n$, which is still not an integer because $r$ is irrational. Thus, $f(x) g_1 (x) \ldots g_n (x)$ is not atomic. $D$ is not almost atomic.

Let $p(x)$ be any non-zero element of $D$. Let $rx^k$ be the first term of $h(x)$. Consider the product $q(x) = (x^2/r) p(x)$. The first term of $q(x)$ is $x^{k + 2}$. Therefore, $q(x) = x^{k + 1} h(x)$. The first term of $h(x)$ is $x$. Every other term has a real coefficient and an exponent greater than $1$. Therefore, $h(x) \in D$. The only constant factors of $h$ are $\pm 1$ because $x$ is the first term of $h(x)$. Let $d$ be the degree of $h$. If $h(x) = h_1 (x) \ldots h_m (x)$, where each $h_i$ is an element of $D$, then $m \leq d$ because $h$ has no nontrivial constant factors. $h(x)$ is atomic. Hence, $D$ is quasi-atomic, but not almost atomic.
\end{example}

Kaplansky's Theorem states that a domain is a UFD if and only every prime ideal contains a prime element. We provide a similar theorem for quasiatomic domains with a similar proof. \cite[Thm. $1.1.5$]{K}

\begin{theorem} Assuming Zorn's Lemma, a domain is quasiatomic if and only if every prime ideal contains an irreducible element.
\end{theorem}

\begin{proof} Let $R$ be a quasiatomic domain with prime ideal $P$. Let $\alpha$ be a non-zero element of $P$. Because $R$ is quasiatomic, there exists some $\beta$ such that $\alpha\beta$ is atomic. $\alpha\beta \in P$. We may write $\alpha\beta$ as $\pi_1 \ldots \pi_n$ with each $\pi_i$ irreducible. $P$ contains some $\pi_i$ because it is a prime ideal.

Let $S$ be the set of atomic elements of $R$. Let $\alpha \in R \backslash (S \cup \{0\})$. Suppose $\alpha R \cap S = \emptyset$. $\alpha R$ is contained in an ideal $P$, which is maximal with respect to $R \backslash S$. $P$ contains an irreducible $\pi$. However, $\pi R \cap S = \emptyset$. But, $\pi \in S$, which is impossible.

Thus, $\alpha R \cap S \neq \emptyset$. Then, there exists some $\beta \in R$ such that $\alpha\beta \in S$. By definition, $\alpha\beta$ is atomic. Therefore, $R$ is quasiatomic.
\end{proof}

\section{Furstenberg Domains}

In 1955, Furstenberg published his ``topological'' proof of the infinitude of primes \cite{F}. Clark recently proved that Furstenberg's proof works in any integral domain in which there are finitely many units and every non-unit element has an irreducible divisor \cite{Cl}. This result leads us to a new definition.

\begin{definition} In a \emph{Furstenberg domain}, every non-unit element has an irreducible divisor.
\end{definition}

It is clear that all atomic domains are also Furstenberg. We no longer demand that elements be factorable into irreducible elements, but merely that they have some irreducible divisor. We can modify each of the definitions in the previous section with this idea in mind. All of the following definitions are original to this paper.

Suppose an $X$-atomic domain $D$ has the property that for all non-unit $\alpha \in D$, there exists some $\beta \in X$ such that $\alpha\beta$ is a product of irreducible elements. We define an $X$-Furstenberg domain as one in which for all non-unit $\alpha$, there exists some $\beta \in X$ such that $\alpha\beta$ has an irreducible divisor $\pi$ that does not divide $\beta$. If we did not impose the condition that $\pi$ could not divide $\beta$, then the definition would be meaningless. We would simply let $\beta$ be irreducible. Then, $\beta$ would be an irreducible divisor of $\alpha\beta$. $X$-atomic implies $X$-Furstenberg.

\begin{definition} In a \emph{semi-Furstenberg domain} $D$, there exists some element $\beta \in D$ such that for any $\alpha \in D\backslash D^*$, there exists some irreducible $\pi \in D$ such that $\pi$ divides $\alpha\beta$, but not $\beta$.
\end{definition}

\begin{definition} $D$ is \emph{almost Furstenberg} if for any $\alpha \in D\backslash D^*$, there exist irreducible elements $\pi, \gamma_1, \gamma_2, \ldots, \gamma_n$ such that $\pi$ divides $\alpha\gamma_1 \gamma_2 \ldots \gamma_n$, but $\pi$ does not divide $\gamma_1 \gamma_2 \ldots \gamma_n$.
\end{definition}

\begin{definition} $D$ is \emph{quasi-Furstenberg} if for any $\alpha \in D\backslash D^*$, there exist $\beta, \pi \in D$ such that $\pi$ is irreducible, $\pi$ divides $\alpha\beta$, and $\pi$ does not divide $\alpha\beta$.
\end{definition}

The weakest condition that we can impose on an integral domain is simply that it has irreducible elements, without assuming any structure on those elements.

\begin{definition} An \emph{antimatter domain} contains no irreducible elements \cite{CDM}.
\end{definition}

\begin{example} Let $D = \F_2 [X]_{(X)}$, where $X = \{x^\alpha : \alpha \in \Q_+\}$. Every element of $D$ has the form $x^\alpha u$, where $\alpha$ is a non-negative rational number and $u$ is a unit. $D$ has no irreducible elements because $x^\alpha = x^\beta x^{\alpha - \beta}$ for all $\beta < \alpha$. Hence, $D$ is an antimatter domain.
\end{example}

\begin{example} We present a domain that is both semi-Furstenberg and almost Furstenberg, but not Furstenberg. Let $D = \bigcup_{n = 1}^\infty (\Z\llbracket x^{1/n} \rrbracket + x^{1 + (1/n)} \Q\llbracket x^{1/n} \rrbracket)$. $D$ is the ring of power series with rational coefficients in which the coefficient of $x^q$ is an integer when $q \leq 1$ and the exponents have bounded denominator.

The irreducible elements of $D$ are precisely those in which the constant term is a rational prime. Note that $x^q$ is never irreducible because $x^q = (x^{q/2})^2$. $x$ is not a multiple of any irreducible element. Consider $xf(x)$, where $f$ is a non-unit element of $D$ with constant term $c(f)$. Let $q$ be the smallest exponent of $f$. If $q$ is positive, then the smallest exponent of $xf(x)$ is $1 + q > 1$. Therefore, $f(x)$ is a multiple of every rational prime, even though $x$ is not a multiple of any rational prime. Suppose $q = 0$. Then, $f(x)$ already has a rational prime divisor because $c(f)$ has such a divisor. $xf(x)$ does as well. Thus, $D$ is semi-Furstenberg.

Because $f$ is not a unit, $c(f) \neq \pm 1$. So, $c(f)$ has a prime factor $p$. Consider $(p + x)f(x)$. $p$ clearly divides $pf(x)$. The first term of $xf(x)$ is $c(f)x$. Every other term has an exponent greater than $1$. Therefore, $p$ divides $xf(x)$ as well. Thus, $p \mid (p + x)f(x)$, even though $p$ does not divide the irreducible element $p + x$. Hence, $D$ is almost Furstenberg.
\end{example}

\section{A Semi-Atomic Domain That Is Not Furstenberg}

Because there are countably many primes, we may associate every pair of positive integers with a distinct odd prime. Let $p_{n, m}$ be the odd prime associated to the pair $(n, m)$ with $m$ odd. Let $\alpha$ be any positive irrational number. Finally, let $S$ be the additively closed set generated by the following elements.
\begin{enumerate}
\item{$\alpha$.}
\item{$1/2^n$ with $n \in \Z$.}
\item{$p_{n, m}^{-1} (\alpha + (m/2^n))$ with $m, n \in \Z_+$ and $m$ odd.}
\end{enumerate}

Let $D = \F_2 [X]_{(X)}$ with $X = \{x^\alpha : \alpha \in S\}$. $D$ is not Furstenberg because $x^{m/2^n}$ has no irreducible divisors for any $m, n \in \Z_+$. We now show that it is semi-atomic.

\begin{lemma} Every element of $D$ has the form $x^q f(x)$ where $q$ is a non-negative rational number and $f(x)$ is atomic.
\end{lemma}

\begin{proof} Suppose $r = p_{n, m}^{-1} (\alpha + (m/2^n))$ was a sum of smaller elements of $S$. $r - \alpha \notin S$ because the coefficient of $\alpha$ is negative. If $q \in \Q_+$, then $r - q \notin S$ because we cannot express $1/p_{n, m}$ as a sum of any other reciprocals of primes. Finally, we cannot subtract any other elements from $(iii)$ because the coefficient of $\alpha$ would be a fraction in which the denominator has multiple prime factors that can only be generated by the primes we used. We cannot express $r$ as a sum of more than one element of $S$.

Let $g(x) = x^{b_1} + \ldots + x^{b_n} \in D$ where each $b_i$ is an element of $S \cup \{0\}$. Let $q_i$ be the largest rational number for which $b_i - q_i \in S$. Such an element exists for $b_i$ because it exists for each generator of $S$. Let $q = \min(q_1, \ldots, q_n)$. Then, $g(x) = x^q f(x)$. One of the exponents of $f(x)$ is $b_i - q_i$. $x^{b_i - q_i}$ is atomic because $b_i - q_i$ is the sum of a multiple of $\alpha$ and various terms of the form $p_{n, m}^{-1} (\alpha + (m/2^n))$, all of which cannot be expressed as a sum of smaller terms. Hence, $f(x)$ is atomic.
\end{proof}

\begin{corollary} $D$ is semi-atomic.
\end{corollary}

\begin{proof} We show that $x^\alpha g(x)$ is atomic for any $g(x) \in D$. We already established that $x^\alpha$ is irreducible. If $g(x)$ is atomic, then so is $x^\alpha g(x)$. Suppose $g(x)$ is not atomic. Then, it has the form $x^q f(x)$ for some $q \in \Q_+$ and atomic $f(x)$. Let $q = m/2^n$ with $m$ odd.
\[x^\alpha x^q = x^{\alpha + (m/n)} = (x^{p_{n, m}^{-1} (\alpha + (m/n))})^{p_{n, m}}\]

In this case, $x^\alpha g(x) = x^{\alpha + q} f(x)$ is atomic.
\end{proof}

\section{An Almost Atomic Domain That Is Not Semi-Furstenberg}

In this construction, each $x_n$, $y_n$, and $z_\gamma$ is an indeterminate. Let $R_0 = \F_2[Y_0]$ with $Y_0 = \{y_0^\alpha : \alpha \in \Q_+\}$. For each non-negative integer $n$, we define $R_n$ as follows.
\[R_{n + 1} = R_n [z_\gamma]\left[\frac{x_{n + 1} \gamma}{z_\gamma}\right]_{\gamma \in \Gamma_n} [Y_{n + 1}]\]
\[Y_n = \{y_n^\alpha : \alpha \in \Q_+\}\]
\[\Gamma_n = \{\alpha \in R_n : \alpha\ \textrm{not atomic in}\ R_n\}\]
\[D = \bigcup_{n = 0}^\infty R_n\]

We prove that $D$ is almost atomic, but not semi-Furstenberg using a few lemmas.

\begin{lemma} $U(R_n) = U(R_{n - 1})$.
\end{lemma}

\begin{proof} Fir a specific $\gamma \in \Gamma_{n - 1}$. Let $u \in U(R_{n - 1} [z_\gamma])$. There exist $a_0, \ldots, a_m$ with the following property.
\[u = a_0 + a_1 z_\gamma + \ldots + a_m z_\gamma^m\]

Because $u$ is a unit, it has an inverse $v$.
\[v = b_0 + b_1 z_\gamma + \ldots + b_k z_\gamma^k\]

$uv = 1 \in R_{n - 1}$. Therefore, $u = a_0$ and $v = b_0$. We see that $u, v \in U(R_{n - 1})$. A similar argument occurs when we adjoin $x_n \gamma/z_\gamma$ in that neither $u$ and $v$ can contain powers of $x_n$. It is not a problem that $\Gamma_{n - 1}$ may be infinite because a polynomial can only contain finitely many indeterminates. A similar process occurs when we adjoin $Y_n$.
\end{proof}

\begin{lemma} $\textrm{Irr} (R_{n - 1}) \subset \textrm{Irr} (R_n)$.
\end{lemma}

\begin{proof} Let $\alpha \in \textrm{Irr} (R_{n - 1})$. Suppose $\alpha = \beta \tau$ with $\beta, \tau \in R_n$. If $\beta$ and $\tau$ are both elements in $R_{n - 1}$, then one of them must be an element of $U(R_{n - 1})$ and therefore $U(R_n)$. Thus, $\beta$ or $\tau$ is an element of $R_n \backslash R_{n - 1}$. This is impossible as it would imply that $\beta\tau$ contains a term with an $x_n$, $y_n$, or $z_\gamma$.
\end{proof}

\begin{lemma} Let $\gamma \in R_{n - 1}$. Then, $\gamma x_n$ is atomic in $D$.
\end{lemma}

\begin{proof} Suppose $\gamma \notin \Gamma_{n - 1}$. Then, $\gamma$ is atomic in $R_{n - 1}$. The irreducible elements of $R_{n - 1}$ are also irreducible in $R_n$. Therefore, $\gamma$ is atomic in $R_n$ as well. $x_n$ is irreducible in $R_n$, making $\gamma x_n$ atomic.

Suppose $\gamma \in \Gamma_{n - 1}$. Then, $\gamma x_n = z_\gamma (\gamma x_n/z_\gamma)$. Once again, $\gamma x_n$ is atomic.
\end{proof}

\begin{lemma} $\nexists \beta \in R_n$ such that $\beta y_n$ has an irreducible divisor that $\beta$ does not have.
\end{lemma}

\begin{theorem} $D$ is almost atomic, but not semi-Furstenberg.
\end{theorem}

\begin{proof} Let $\gamma \in D$. Then, there exists some $n \in \Z_+$ such that $\gamma \in R_n$. $\gamma x_{n + 1}$ is atomic by Lemma $21$. Therefore, $D$ is almost atomic. Suppose $D$ is semi-Furstenberg. Then, there exists some $\beta \in D$ such that if $\alpha \in D$, then $\alpha\beta$ has an irreducible divisor that $\alpha$ does not. Once again, there exists some $n$ for which $\beta \in R_n$. Let $\alpha = y_n$. Then, every irreducible divisor of $\beta y_n$ also divides $\beta$.
\end{proof}

\section{Two Final Domains}

In this section, we obtain examples of the last two items of our list in the introduction. Because they are short proofs, they do not get their own sections.

\begin{lemma} Not all Furstenberg domains are quasi-atomic.
\end{lemma}

\begin{proof} Let $D = \Z + x\Q[x]$. In other words, $D$ is the ring of polynomials with rational coefficients in which the constant term is an integer. The irreducible element of $D$ are the rational primes and the polynomials with constant term $1$ that are irredudicble in $\Q[x]$ \cite{BC}. Let $\alpha \in D$ and let $c$ be the constant term of $\alpha$. If $c \neq \pm 1$, then $c$ has a rational prime divisor $p$. So, $p$ divides $\alpha$. If $c = \pm 1$, then $c$ has no non-unit constant divisors. If we write $\alpha$ as a product of non-unit elements, then the number of terms in that product is at most the degree of $\alpha$. In this case, we can express $\alpha$ as a product of irreducible elements. $D$ is Furstenberg.

Let $p$ be a rational prime. If $p$ is a product of two elements $\alpha, \beta \in D$, then $\alpha$ and $\beta$ have degree zero. Therefore, $\alpha$ and $\beta$ are integers. Either $\alpha$ or $\beta$ is $\pm p$ because $p$ is prime. The rational primes are irreducible in $D$.

Let $f$ be a polynomial with constant term $0$. Then, $f(x) = p(f(x)/p)$ for any prime $p$. $f(x)/p$ is an element of $D$ because its constant term is zero and all of its other terms are rational. Therefore, $f$ cannot be factored into irreducible elements. Every multiple of $f$ cannot be factored into irreducible elements because every multiple of $f$ has a constant term of $0$. $D$ is not quasi-atomic.
\end{proof}

\begin{lemma} There are domains that are neither antimatter nor quasi-Furstenberg.
\end{lemma}

\begin{proof} Let $D = \bigcup_{n = 1}^\infty \Z[x^{1/n}]$. Note that $D$ is the same as the domain from Example $12$ except we are allowing the constant term to be nonzero. $D$ is not an antimatter domain because the rational primes are irreducible.

Let $\alpha = x$ and let $\beta$ any nonzero element of $D$. We can show that if an irreducible element $\pi$ divides $\alpha\beta$, then $\pi$ must divide $\beta$ as well. If $\pi$ has constant term $0$, then $\pi$ is a multiple of $x^q$ for some $q \in \Q_+$. But, $x^q$ is not irreducible. Therefore, the constant term of $\pi$ is $0$. But, the constant term of $\beta x$ is $0$. Hence, $\pi x \mid \beta x$ and $\pi \mid \beta$. $D$ is not quasi-Furstenberg.
\end{proof}

\section*{Acknowledgements}
This paper came out the University of Georgia 2015-2016 AGANT group. Funding support came from National Science Foundation RTG grant DMS-1344994. This project began in a 2015-2016 VIGRE group run by Paul Pollack and Pete L. Clark. In addition, the author wishes to thank Jim Coykendall, Daniel Krashen, and Ziqing Xiang for their help in writing this paper.

\begin{appendix}
\section{Another Semi-Atomic Domain}\label{App:AppendixA}

In this section, we present another semi-atomic domain that is not also atomic. While this domain shares certain similarities with the domain from Section $4$, it is still different enough to warrant attention in and of itself.

Let $r_1, r_2, \ldots$ be a countable set of $\Q$-linearly independent real numbers with limit $0$ and sum $1$. Let $K$ be the subset of $(\Z_+ \cup \{0\})^{\Z_+}$ in which all but finitely many terms have the same value. We define an injection from $K$ to $\R$.
\[(a_1, a_2, \ldots) \mapsto \sum_{i = 1}^\infty a_i r_i\]

Because the sequence on the left is bounded and the sum of the $r_i$'s is $1$, the map is well-defined. It is injective because the $r_i$'s are linearly independent over $\Q$. Let $S$ be the additively closed subset of $K$ generated by elements of the following two forms.
\begin{enumerate}
\item{Every $a_i$ is a multiple of $7$ and not all $a_i$ are $0$.}
\item{$\lim_{i \to \infty} a_i \in \{3, 5\}$.}
\end{enumerate}

For a given element $r = (a_1, a_2, \ldots) \in S$, we refer to $\lim_{i \to \infty} a_i$ as the limit $L(r)$ of $S$. Let $e_n$ be the element in which $a_n = 1$ and $a_i = 0$ when $i \neq n$.

Let $X = \{x^\alpha : \alpha \in S\}$ and $R = \F_2 [X]_{(X)}$. $\F_2 [X]$ is the set of polynomials in one variable $x$ in which the exponents are elements of $S$. $R$ is the localization in which the polynomials with constant term $1$ have inverses. $\F_2 [X]$ is a ring because $S$ is closed under addition. $(X)$ is a maximal ideal because $\F_2 [X]/(X) \cong \F_2$, which a field. Thus, $R$ is well-defined.

Let $M = S \backslash (S + S)$ and $\langle M \rangle$ be the set of finite sums of elements of $M$.

\begin{lemma} $m \in M$ if and only if $m$ satisfies one of the following two conditions.
\begin{enumerate}
\item{$m = 7e_n$ for some $n \in \Z_+$.}
\item{$L(m) \in \{3, 5\}$ and every term in $m$ is less than $7$.}
\end{enumerate}
\end{lemma}

\begin{proof} If $m = 7e_n$, then $m - r \notin S$ for all $r$ with limit $3$ or $5$. In addition, $m$ cannot be the sum of two nonzero elements in which every term is a multiple of $7$. Therefore, $m \in M$. If $L(m)$ is either $3$ or $5$ and every term in $m$ is less than $7$, then $m - 7e_n \notin S$ for all $n \in \Z_+$. In addition, $m - r \notin S$ for all $r$ with limit $3$ or $5$ because $L(m - r) < 3$, even though every term is less than $7$.

Suppose $m = (a_1, a_2, \ldots) \in M$. $M$ is generated by elements with limit $0$, $3$, and $5$. So, $L(m) \in \{0, 3, 5\}$. Suppose $L(m) \in \{3, 5\}$. If $a_n > 6$, then $m - 7e_n \in S$ and $m \notin M$. Hence, every $a_n$ is at most $6$. Suppose $L(m) = 0$. It is impossible for $m - r$ to be an element of $S$ if $L(r) > 0$ because that would imply that $L(m - r)$ is negative. Thus, every $a_n$ is a multiple of $7$. If $a_n > 7$, then $m - 7e_n \in S$. If $a_n, a_m \neq 0$ for distinct $n, m$, then $m - 7e_n \in S$ is nonzero. If $L(m) = 0$, then $m = 7e_n$.
\end{proof}

\begin{lemma} The only elements of $S$ that do not belong to $\langle M \rangle$ are those with limit $7$.
\end{lemma}

\begin{proof} Let $r = (a_1, a_2, \ldots) \in S$. We consider the possible values of $L(r)$.
\begin{enumerate}
\item{$L(r) = 0$. If some $a_i$ is not a multiple of $7$, then there exists some $m$ with limit $3$ or $5$ such that $r - m \in S$. This is impossible because $L(r) = 0$. Every $a_i$ is a multiple of $7$ and there are only finitely many nonzero $a_i$. We may write $r$ as a finite sum $e_{b_1} + \ldots + e_{b_n}$ where $b_1, \ldots, b_n$ form a sequence of not necessarily distinct positive integers. $r \in \langle M \rangle$.}
\item{$L(r)$ is $3$ or $5$. Only finitely many terms of $r$ are greater than $6$. We can subtract elements of the form $7e_n$ from $r$ until every term is smaller than $7$. In the process, we have only subtracted elements of $M$ and obtained an element of $M$ as the result. Once again, $r \in \langle M \rangle$.}
\item{$L(r) = 6$. $r = r_1 + r_2$ in which $L(r_1) = L(r_2) = 3$. Only finitely many terms are greater than $7$ and we can subtract $7e_n$ for each one.}
\item{$L(r) > 7$. Every number greater that $7$ can be expressed in the form $3x + 5y$ with $x, y \in \Z_+ \cup \{0\}$ because $3$ and $5$ are relatively prime and $(3 \cdot 5) - (3 + 5) = 7$. Either $x$ or $y$ is positive. Suppose $x$ is positive. We may subtract $x - 1$ elements with limit $3$ and $y$ elements with limit $5$ from $m$ and obtain an element ending in $3$. This reduces to the previous case. If $y$ is positive, we obtain a sequence ending in $5$ and the proof is similar. $r \in \langle M \rangle$ again.}
\item{$L(r) = 7$. Suppose $r \in \langle M \rangle$. Then, there exist $m_1, \ldots, m_n \in M$ such that $r = m_1 + \ldots + m_n$. Each $m_i$ has limit $0$, $3$, or $5$ by the previous lemma. But, it is impossible to express $7$ in the form $3x + 5y$ with $x$ and $y$ non-negative integers. Therefore, $r \notin \langle M \rangle$.}
\end{enumerate}
\end{proof}

\begin{theorem} $R$ is not atomic.
\end{theorem}

\begin{proof} Let $\alpha \in S \backslash \langle M \rangle$. $\alpha$ is well-defined by the previous lemma. Consider the element $x^\alpha \in R$. Suppose $x^\alpha$ were atomic. Then, there would exist a finite set of irreducible polynomials $f_1 (x), \ldots, f_n (x) \in R$ such that $x^\alpha = f_1 (x) \ldots f_n (x)$. Each $f_i$ has the form $x^\beta_i$ for some $\beta_i \in R$. We have $\alpha = \beta_1 + \ldots + \beta_n$. Because each $f_i$ is irreducible, $f_i \in M$. Therefore, $\alpha = \beta_1 + \ldots + \beta_n \in \langle M \rangle$, contradicting our original assumption. Hence, $R$ is not atomic.
\end{proof}

In order to show that $R$ is not semi-atomic, we will prove a lemma about the nature of $R$ first.

\begin{lemma} Every element of $R$ is atomic or of the form $x^\beta g(x)$, where $g(x)$ is atomic and $\beta$ is an element of $S$ with limit $7$.
\end{lemma}

\begin{proof} Let $f(x) \in R$. Let $f(x) = x^{\alpha_1} + \ldots + x^{\alpha_n}$, where each $\alpha_i$ is a distinct element of $S$. (If some $\alpha_i$ is zero, then $f$ is a unit.)
\begin{enumerate}
\item{There exists some $i \leq n$ such that $L(\alpha_i)$ cannot be expressed in the form $\alpha_i = 3x + 5y + 7z$ with $x$ and $y$ non-negative and $z$ positive. The only possible values for $L(\alpha_i)$ are $0$, $3$, $5$, and $6$. In each case, we cannot express $\alpha_i$ as a sum of more than $C$ elements for some constant $C$. Therefore, we cannot write $f(x)$ as a product of more than $C$ elements. $f$ is atomic.}
\item{For every $i$, there exist non-negative $x$ and $y$ and positive $z$ such that $\alpha_i = 3x + 5y + 7z$. We may keep dividing $f(x)$ by elements of the form $x^\beta$ in which the limit of $\beta$ is $7$. There exist $\beta_1, \ldots, \beta_m$ with limit $7$ such that $f(x) = x^{\beta_1} \ldots x^{\beta_m} g(x)$, where $g(x)$ is an element of the previous case. So, $g(x)$ is atomic. If $m > 1$, then $x^{\beta_1 + \ldots + \beta_m}$ is atomic because $\beta_1 + \ldots + \beta_m \in \langle M \rangle$. If $m > 1$, then $f$ is atomic. If $m = 1$, then $f(x) = x^\beta g(x)$, where $L(\beta) = 7$ and $g$ is atomic.}
\end{enumerate}
\end{proof}

\begin{theorem} $R$ is semi-atomic.
\end{theorem}

\begin{proof} Let $\alpha$ be an element of $S$ with limit $3$ and no terms greater than $6$. We show that $x^\alpha f(x)$ is atomic for any $f(x)\in R$. Note that $x^\alpha$ is irreducible by Lemma $14$. If $f(x)$ is atomic, then so is $x^\alpha f(x)$.

Suppose $f$ is not atomic. By Lemma $16$, there exists some $\beta \in S$ with limit $7$ and atomic $g(x) \in R$ such that $f(x) = x^\beta g(x)$. So, $x^\alpha f(x) = x^{\alpha + \beta} g(x)$. $\alpha + \beta$ has limit $10$. By Lemma $15$, $\alpha + \beta \in \langle M \rangle$. There exist $m_1, \ldots, m_n \in M$ such that $\alpha + \beta = m_1 + \ldots + m_n$. In this case, $x^\alpha f(x) = x^{m_1} \ldots x^{m_n} g(x)$. It is possible to express $x^\alpha f(x)$ as a product of irreducible elements.
\end{proof}

\section{A Correction}\label{App: Appendix B}

Let $D = \F_2 [X, y, Z]_{(X, y, Z)}$, where $X = \{x^\alpha : \alpha \in \Q_+\}$ and $Z = \{x^\alpha y^k : \alpha \in \Q, k \in \Z, k > 1\}$. We prove that $D$ is almost atomic, but not Furstenberg. Every element of $D$ is the product of a polynomial and a unit. For convenience, we only consider the polynomial parts of each element.

\begin{lemma} The irreducible polynomials in $R$ are precisely those in which the coefficient of $y$ is $1$ and the constant term is $0$.
\end{lemma}

\begin{proof} Let $f(x, y)$ be a polynomial in $R$. If the constant term of $f$ is $1$, then $f$ is a unit. Suppose the constant term is $0$. If the coefficient of $y$ is not $1$, then there exists some $\alpha \in \Q_+$ such that $x^\alpha$ divides $f(x)$. Of course, $x^\alpha$ is not irreducible or a unit. Suppose $y$ is a term in $f$. Suppose $f(x, y) = g(x, y)h(x, y)$, where $g$ and $h$ are polynomials in $R$. The product of one of the terms of $g$ and one of the terms of $h$ is $y$. Therefore, $1$ is a term of either $g$ or $h$, making either $g$ or $h$ a unit. If we multiply $f$ by a unit, one of the terms is still $y$. In other words, if we write $f$ as the product of two elements of $R$, then one of those elements is a unit. $f$ is irreducible.
\end{proof}

\begin{corollary} Let $f(x, y)$ be a polynomial in $R$. If there exists a non-negative integer $k$ such that $y^{k +1}$ is a term in $f$ and $f$ is a multiple of $y^k$, then $f$ is atomic.
\end{corollary}

\begin{proof} Because $f$ is a multiple of $y^k$, there exists some $g(x, y) \in R$ such that $f(x, y) = y^k g(x, y)$. One of the terms of $g$ is $y$. Therefore, $g$ is irreducible and $f$ is a product of irreducible elements.
\end{proof}

\begin{theorem} $R$ is almost atomic.
\end{theorem}

\begin{proof} Let $f(x, y) \in R$. $f(x, y)$ is the product of a polynomial and a unit. We just consider the polynomial part. Let $x^\alpha y^k$ be the first term of $f$ in the lexicographic ordering on $\Z \oplus \Q$. In other words, we choose $k$ to be as small as possible. If there is a tie, we choose the term with the smallest $\alpha$. There are a few cases.
\begin{enumerate}
\item{$\alpha = 0$. $y^k$ is the smallest term of $f(x, y)$. The exponent of $y$ in any term of $f$ is at least $k$. If $k = 0$, then $f$ is a unit. Otherwise, every term is a multiple of $y^{k - 1}$, making $f$ atomic by Corollary $1$.}
\item{$\alpha < 0$. By assumption, $k > 1$. Consider $(y + x^{-\alpha})f(x, y)$. If the exponent of $y$ in a given term is $r$ with $r \geq 2$, then that term is a multiple of $y^{r - 2}$. Thus, every term in $f(x, y)$ is a multiple of $y^{k - 2}$. Therefore, $y^{k - 1}$ divides $yf$. The smallest term in $x^{-\alpha} f(x, y)$ is $y^k$. Every term in $x^{-\alpha} f$ has the form $x^\beta y^k$ with $\beta \geq 0$ or $x^\beta y^r$ with $r > k$. Either way, $y^{k - 1}$ divides $x^\beta y^k$. Hence, $y^{k - 1}$ divides $x^{-\alpha} f$. Thus, $y^{k - 1} \mid (y + x^{-\alpha}) f$, which has $y^k$ as a term. $(y + x^{-\alpha})f(x, y)$ is atomic.}
\item{$\alpha > 0$ and $f(x, y)$ does not contain a term of the form $x^{-\beta} y^{k + 1}$ with $\beta \in \Q_+$. Consider $(y + x^{-\alpha} y^2) f(x, y)$. Every term in $f(x, y)$ has the form $x^\gamma y^k$ with $\gamma \geq 0$, $x^\gamma y^{k + 1}$ with $\gamma \geq 0$, or $x^\gamma y^r$ with $\gamma \in \Q$ and $r > k + 1$. All of these terms are multiples of $y^k$. Therefore, $y^{k + 1}$ divides $yf(x, y)$. The smallest term in $x^{-\alpha} y^2 f(x, y)$ is $y^{k + 2}$. Hence, $y^{k + 1}$ divides $x^{-\alpha} y^2 f(x, y)$ as well. $(y + x^{-\alpha} y^2) f(x, y)$ is almost atomic because it is a multiple of $y^{k + 1}$ and has $y^{k + 2}$ as one of its terms.}
\item{$\alpha > 0$ and $f(x, y)$ contains a term of the form $x^{-\beta} y^{k + 1}$ with $\beta \in \Q_+$. Choose $\beta$ to be the largest rational number in which $x^{-\beta} y^{k + 1}$ is a term of $f$. Consider $(y + x^\beta) f(x, y)$. Every term in $x^\beta f(x, y)$ has the form $x^\gamma y^k$ with $\gamma > 0$, $x^\gamma y^{k + 1}$ with $\gamma \geq 0$, or $x^\gamma y^r$ with $\gamma \in \Q$ and $r > k + 1$. All of these terms are multiples of $y^{k + 1}$. All of these terms are multiples of $y^k$. Every term in $yf(x, y)$ has the form $x^\gamma y^{k + 1}$ with $\gamma > 0$ or $x^\gamma y^r$ with $r > k + 1$. Once again, these terms are all multiples of $y^k$. Hence, $y^k \mid (y + x^\beta) f(x, y)$. Also, $(y + x^\beta) f(x, y)$ has $y^{k + 1}$ as a term.}
\end{enumerate}
\end{proof}

In \cite[Ex. $5.5.1$]{BC}, the authors claim that $D$ is not almost atomic. But, we have just shown that this is incorrect. They obtained this result due to a mistake in their Example $4.4.1$. They claimed that the irreducible elements of $D$ are precisely those in which the smallest term is $y$ when the terms are put in lexicographic ordering as elements of $\Z \oplus \Q$. In fact, the irreducible elements are simply those which contain $y$ as an element. For example, $y + x^\alpha$ comes before $y$ in the lexicographic ordering for any $\alpha \in \Q_+$. Unfortunately, $D$ is the only example of a quasi-atomic domain that is not almost atomic in \cite{BC}. Example $9$ in this paper is quasi-atomic, but not almost atomic.
\end{appendix}

\end{document}